\documentclass{amsart}

\usepackage{mabliautoref}
\usepackage{joe-custom}

\newcommand\psef{\operatorname{pseudo-effective}}

\begin{document}

\title{Mori fibre spaces for $3$-folds over imperfect fields}
\author{Joe Waldron}
\address{Department of Mathematics, Michigan State University, East Lansing, MI 48824, USA}
\email{waldro51@msu.edu}

\begin{abstract}
We prove that the LMMP works for projective threefolds over function fields of characteristic $p>5$ when the canonical divisor is not pseudo-effective.  In the process we show that ACC for log canonical thresholds holds in complete generality for threefolds in residue characteristic $p>5$. 
\end{abstract}

\maketitle
\tableofcontents

\section{Introduction}

Over the last decade, the Log Minimal Model Program (LMMP) has been nearly completed in great generality for klt pairs of dimension three in positive and mixed characteristic $p>5$.   Firstly it was shown to hold for three dimensional varieties which are projective over an algebraically closed field \cite{hacon_three_2015, cascini_base_2015, birkar_existence_2016, birkar_existence_2017}, then extended to perfect ground fields \cite{gongyo_rational_2015-1}, and to varieties over $F$-finite ground fields for which $K_X+B$ is pseudo-effective \cite{das_waldron_imperfect}.  Most recently it was shown to hold over any finite dimensional excellent scheme with dualising complex, so long as the image of $X$ in the base has positive dimension \cite{bhatt_globally, takamatsu_yoshikawa}.   This paper aims to complete the picture: to deal with the case of klt pairs $(X,B)$ which are projective over an imperfect field and such that $K_X+B$ is not pseudo-effective.  Following this paper, the LMMP is essentially complete for klt schemes of dimension $3$ in residue characteristic $p>5$.

Most of our effort will be directed towards proving:

\begin{theorem}[Base point free theorem]\label{thm:bpf}
Let $(X,B)$ be a $\mathbb{Q}$-factorial three dimensional klt pair, projective over a field of characteristic $p>5$, and let $L$ be a nef $\mathbb{R}$-divisor such that $L-(K_X+B)$ is nef and big.  Then $L$ is semi-ample.
\end{theorem}

This is the main ingredient needed for the existence of Mori fibre spaces for $3$-folds over imperfect fields of characteristic $p>5$:

\begin{theorem}[Existence of Mori fibre spaces]\label{thm:mfs}
Let $(X,B)$ be a three dimensional $\mathbb{Q}$-factorial klt pair, projective over a field of characteristic $p>5$.  Then we can run a $(K_X+B)$-MMP which terminates with either Mori fibre space $f:Y\to Z$ or a log minimal model.
\end{theorem}

These results were proved over algebraically closed fields in \cite{birkar_existence_2017}.  Note that our results are not stated in a form which includes the relative situation (as in \cite{birkar_existence_2017}), since that is already covered by \cite{bhatt_globally}, and indeed will be used here in several places.  Note that in order to obtain a log minimal model or Mori fiber space for a non-$\mathbb{Q}$-factorial pair, one can take a small $\mathbb{Q}$-factorialization using \cite[Corollary 9.21]{bhatt_globally} and run the MMP from there using the above statement. 

These results are motivated by the possibility of running a terminating LMMP on the generic fibres of contractions of relative dimension $3$, which has potential applications to the LMMP in higher dimensions.  The following is an immediate consequence in dimension $4$ which follows from the work of Xie and Xue \cite{xie_xue}:

\begin{corollary}\cite[Theorem 1.1]{xie_xue}
	Let $R$ be an excellent discrete valuation ring of characteristic $p>5$ with perfect residue field, and let $s\in\Spec(R)$ be the closed point.  		Assume also that log resolutions given by a sequence of blowups in the non-snc locus exist for schemes of dimension $4$ which are of finite type over $R$.
	
	Let $(X,\Delta)$ be a $\mathbb{Q}$-factorial dlt pair of dimension $4$, projective over $\Spec(R)$ such that $\Supp(X_s)\subseteq \rddown{\Delta}$.  
	Then we can run a $(K_X+\Delta)$-MMP over $\Spec(R)$ with scaling of an ample divisor $A$, and it terminates. 
\end{corollary}

As stated by Xie-Xue (\cite[Comment before Theorem 1.2]{xie_xue}), the only reason that this case of their theorem was excluded from \cite{xie_xue} was the lack of the threefold LMMP over the generic point of $\Spec(R)$.

In the course of our argument we prove several ascending chain condition (ACC) results which are interesting in their own right.  These results are tightly linked via the local-global principle, as was the case in Hacon-M\textsuperscript{c}Kernan-Xu's proof of ACC for log canonical thresholds in characteristic zero \cite{hacon_acc_2014}.
Note that the corresponding result for $X$ of dimension $1$ is \cite[Proposition 5.4]{das_boundedness}.

\begin{theorem}[Global ACC for surfaces]\label{thm:global_acc}
Let $\Lambda\subset [0,1]$ be a DCC set of real numbers.  Then there is a finite subset $\Gamma\subset\Lambda$ with the following property:
Let $(X,B)$ be a pair projective over a field $K$ of 
 characteristic not equal to $2$, $3$ or $5$ such that 
\begin{itemize}
\item $\dim(X)=2$
\item $(X,B)$ is dlt
\item The coefficients of $B$ are in $\Lambda$ 
\item $K_X+B\equiv_Z 0$ 
\end{itemize}
Then the coefficient of each component of $B$ is in $\Gamma$.
\end{theorem}

\begin{theorem}[ACC for log canonical thresholds]\label{thm:acc_for_lct}
Suppose that $\Lambda\subset[0,1]$ and $\Gamma\subset \mathbb{R}$ are DCC sets.  
Let $\mathfrak{L}_3(\Lambda)$ denote the set of excellent log canonical pairs of dimension $3$ with residue characteristic not equal to $2$, $3$ or $5$, whose coefficients lie in $\Lambda$.  Let $\mathfrak{G}(\Gamma,X)$ denote the set of $\mathbb{Q}$-Cartier divisors on a fixed $X\in \mathfrak{L}_3(\Lambda)$ whose coefficients lie in $\Gamma$.

Then the set 
\[\{\mathrm{lct}(M,X,B)\mid (X,B)\in\mathfrak{L}_3(\Lambda),\ M\in\mathfrak{G}(\Gamma,X)\}\]
satisfies the ACC.
\end{theorem}

Note that the singularities in question could be over a field or arithmetic base, and the sets in the statement do not depend on the base.  The analogues of these statements for threefolds over algebraically closed fields of characteristic $p>5$ were proved in \cite{birkar_existence_2016}.

ACC for log canonical thresholds has an immediate corollary to termination of log canonical threefold flips, using the argument of \cite{birkar_acc_2005}.  Here we include the relative situation since this is also new.

\begin{corollary}[Termination of effective lc flips]\label{cor:lc_termination}
Let $T$ be a scheme which is quasi-projective over a finite dimensional excellent ring admitting a dualising complex and no residue characteristics $2$, $3$ or $5$.  Suppose that $(X,B)$ is a log canonical threefold pair with $\mathbb{R}$-boundary which is projective over $T$.  

  Then any sequence of $(K_X+B)$-flips which are also $M$ flips for some $M\geq 0$ terminates.  
\end{corollary}

\subsection{Outline of the proof}

The main result of the paper is the base point free theorem \autoref{thm:bpf} in the case where $L$ is not big, the other case having already been dealt with in \cite{das_waldron_imperfect}.  We first prove the ACC results by induction on the transcendence degree of the ground field in \autoref{sec:acc}, before moving on the base point free theorem.   As in \cite{birkar_existence_2017}, we first approximate the semi-ample fibration of $L$ using the nef reduction map $f:X\dashrightarrow Z$, which we construct in \autoref{sec:nef_red}.  We must then prove that $L$ pulls from some divisor $D$ on $Z$.  The most crucial part of the proof of the base point free theorem is to show that $D$ is big, and this is where we require ACC, and it is carried out in \autoref{sec:goodness}.  Given this, the rest of the proof in \autoref{sec:bpf} is as in the relative case \cite[Section 9]{bhatt_globally}, where we show that the big divisor $D$ is semi-ample as an application of Keel's theorem.  Finally in \autoref{sec:non_F-finite} we remove all conditions on $F$-finiteness of the ground field.

\subsection{Acknowledgements}
The author would like to thank Shikha Bhutani for pointing out errors in an earlier version and the referee for their valuable comments and suggestions.  The author was supported by NSF CAREER Grant DMS \#2440240, NSF Grant DMS \#2401279 and the Simons Foundation Gift ID \#850684
\section{Preliminaries}

We work over an arbitrary field $K$ of characteristic $p>5$, which we sometimes assume to be $F$-finite. When we require the ground field to be perfect, we will denote it by lower case $k$.  
A \emph{variety} over a field $K$ means an integral scheme which is separated and of finite type over $K$.  In particular, we do not assume varieties to be geometrically integral.

For definitions of singularities of the LMMP and for concepts such as Mori fibre space, log minimal model etc., we refer to \cite[Subsection 2.5]{bhatt_globally}.  Note that klt, dlt etc are defined without reference to a log resolution.  Boundaries will be $\mathbb{R}$-boundaries unless otherwise specified. 

\subsection{Resolution of singularities and Bertini theorems}

We use resolution of singularities in the following form:

\begin{theorem}[\cite{cossart_resolution_2014}, \cite{CossartJannsenSaito}, \cite{kollar_witaszek}]\label{thm:resolution}
Let $X$ be a reduced scheme of dimension $3$, quasi-projective over a Noetherian quasi-excellent affine scheme $\Spec(R)$.  Let $T$ be a subscheme of $X$.  Then there exist a birational morphism $g:Y\to X$ from a regular scheme $Y$ such that both $g^{-1}(T)$ and $\Ex(g)$ are divisors, $\Supp(g^{-1}(T)\cup\Ex(g))$ is simple normal crossing and $Y$ supports a $g$-ample $g$-exceptional divisor.
\end{theorem}

\begin{proof}
The theorem is stated and proved in this form as \cite[Proposition 2.14]{bhatt_globally}.
\end{proof}

The Bertini theorem in \cite{bhatt_globally} is sufficient to obtain:

\begin{lemma}[{\cite[Lemma 2.34]{bhatt_globally}}]\label{lem:add_ample}
Let $(X,B)$ be a klt pair which is projective over a field, and let $A$ be an ample divisor on $A$.  Then there exists $0\leq A'\sim_{\mathbb{Q}}A$ such that $(X,B+A')$ is klt.
\end{lemma}

This usually has a variant for dlt pairs, however due to the limitations of resolution available in our situation, we are not able to use this unless we additionally assume that the ground field is $F$-finite.

\subsection{Semi-ampleness in  positive characteristic}

In this subsection, we give some criteria for semi-ampleness which are unique to fields of positive characteristic, and which will be heavily used throughout. The first is due to Keel.

\begin{definition}
Let $L$ be a nef line bundle on a scheme $X$ which is proper over a field.  A irreducible subvariety $Z\subset X$ is \emph{exceptional} for $L$ if $L|_Z$ is not big (i.e. $L^{\dim(Z)}\cdot Z=0$).  We denote the union of all exceptional subvarieties by $\mathbb{E}(L)$.
\end{definition}

\begin{definition}
We say that a nef divisor $D$ on a proper variety $X$ is \emph{endowed with map (EWM)} if there exists a proper map $f:X\to Y$ to a proper algebraic space such that $f$ contracts a subvariety $Z$ if and only if $Z$ is exceptional for $L$.  
\end{definition}

\begin{theorem}\cite[Theorem 1.9]{keel_basepoint_1999}\label{thm:keel}
Let $L$ be a nef line bundle on a scheme $X$ which is projective over a field of characteristic $p>0$. 
Then $L$ is semi-ample (resp. EWM) if and only if $L|_{\mathbb{E}(L)}$ is semi-ample (resp. EWM).
\end{theorem}

Keel's theorem is only useful if the line bundle $L$ is big.  However, the next criterion due to Cascini and Tanaka can sometimes be used if it is not.

\begin{theorem}\cite[Theorem 1.1]{cascini_relative_2017}\label{thm:cascini_tanaka}
Let $f:X\to S$ be a projective morphism of Noetherian $\mathbb{F}_p$-schemes.  Let $L$ be an invertible sheaf on $X$ such that $L|_{X_s}$ is semi-ample for every point $s\in S$.  Then $L$ is $f$-semi-ample.
\end{theorem}

Note that it is not sufficient to check only closed points in \autoref{thm:cascini_tanaka}.

\subsection{Adjunction coefficients}

For our proof of the ACC results, we need to have careful control of the coefficients which appear under adjunction. 

\begin{proposition}\label{prop:adjunction_coefficients}
Let $(X,B)$ be a log canonical pair, $S$ a component of $\rddown{B}$, $S^{\eta}\to S$ be the normalisation and $B_{S^\eta}$ the boundary defined by adjunction \cite[Section 4.2]{kollar_singularities_2013}.  Write $B=S+\sum b_i B_i$ with $B_i$ prime divisors.  Let $V$ be a prime divisor on $S^{\eta}$.  Then there exists $m\in\mathbb{N}\cup\{\infty\}$ and $l_i\in\mathbb{N}\cup\{0\}$ such that the coefficient of $V$ in $B_{S^\eta}$ is 
\[\frac{m-1}{m}+\sum\frac{l_ib_i}{m}\]

Hence if $B$ has coefficients from a DCC set $\Lambda$, then $B_{S^\eta}$ has coefficients from a DCC set $\mathfrak{G}(\Lambda)$ depending only on $\Lambda$.
\end{proposition}

\begin{proof}
The proof in \cite{birkar_existence_2017} immediately reduces to the case of an excellent surface, so this proof also works in arbitrary dimensions and more general excellent schemes.
\end{proof}

\subsection{Log minimal model program}

In this subsection we summarise the known results of the threefold LMMP which we need to apply later.

The first is the relative LMMP of \cite{bhatt_globally}.  Note that this applies very generally, and we will use the full power of the LMMP over non-quasi-projective schemes in our proof of ACC. 

\begin{theorem}\cite[Theorem 9.35]{bhatt_globally}\label{thm:relative_mmp}
Let $T$ be a scheme which is quasi-projective over a finite dimensional excellent ring $R$ admitting a dualising complex, none of whose residue characteristics are $2$, $3$ or $5$.  

Let $(X,B)$ be a three dimensional $\mathbb{Q}$-factorial klt pair with $\mathbb{R}$-boundary, which is projective over $T$, such that the image of $X$ in $T$ has positive dimension.  Suppose that $A$ is an ample divisor on $X$ such that $K_X+B+A$ is nef over $T$.  Then we can run the $(K_X+B)$-MMP over $T$ with scaling of $A$ and it terminates with either a log minimal model or Mori fibre space.
\end{theorem}

The next set of results are from \cite{das_waldron_imperfect}, which developed the LMMP over imperfect fields for pseudo-effective threefold pairs $(X,B)$.  Our main interest here is to deal with the non-pseudo-effective case.  First the cone theorem, which is crucial input to running a terminating LMMP in the non-pseudo-effective case.  We give a slightly strengthened version of the original statement, which we will need later.

\begin{theorem}\cite[Theorem 1.1]{das_waldron_imperfect}\label{thm:cone}
	Let $(X,B)$ be a projective $\bQ$-factorial, dlt pair of dimension at most $3$, over a field $K$ of characteristic $p> 5$. Then there is a countable set of curves $\{C_i\}_{i\in I}$ such that:
		\begin{enumerate}
			\item \[\overline{NE}(X)=\overline{NE}(X)_{K_X+B\geq 0}+\sum_{i\in I}\mathbb{R}_{\geq 0}[C_i]\]
			\item For any ample $\bQ$-divisor $A$, there is a finite subset $I_A\subset I$ such that
			\[\overline{NE}(X)=\overline{NE}(X)_{K_X+B+A\geq 0}+\sum_{i\in I_A}\mathbb{R}_{\geq 0}[C_i]\]
			\item The rays $[C_i]$ do not accumulate in $\overline{NE}(X)_{K_X+B<0}$.
			\item  For each $C_i$, there is a unique positive integer $d_{C_i}$
			depending only on $X, C_i$ and the ground field $K$ satisfying the following properties: 
			\begin{enumerate}
			\item If $C^{\overline{K}}_i$ is an integral curve on $X\otimes_K\overline{K}$ which maps to $C_i$ via $\phi:X\otimes_K\overline{K}\to X$, then for any Cartier divisor $L$ on $X$ we have 
			\[d_{C_i}=\frac{L\cdot_K C_i}{\phi^*L\cdot_{\overline{K}}C^{\overline{K}}_i}.\]
			In particular, $L\cdot_K C_i$ is divisible by $d_{C_i}$.
\item			\( 0<-(K_X+B)\cdot_K C_i\leq  6d_{C_i}\)

\end{enumerate}
		\end{enumerate}
		
Furthermore, if $Y$ is the normalisation of $X\otimes_K \overline{K}$, we can additionally choose the curves $C_i$ to be the images of rational curves from $Y$.
\end{theorem}

\begin{proof}
	The first change from \cite{das_waldron_imperfect} is to drop the $F$-finite assumption.  This can be obtained simply by referring to the proof of \cite[Theorem 9.28]{bhatt_globally} in place of the proof of \cite[Theorem 1.1]{birkar_existence_2017} in \cite[Proof of Theorem 1.1]{das_waldron_imperfect}.
	
	The formula for $d_{C_i}$ in (4.a) is \cite[Lemma 4.1]{das_waldron_imperfect}.
	
Now it only remains to prove the furthermore, since everything else is covered in \cite{das_waldron_imperfect}.  
In fact, the curves found in \cite{das_waldron_imperfect} already had the required property, so it remains to work through the proof in \cite{das_waldron_imperfect} to verify this.  The proof of the cone theorem is split into several pieces:
\begin{itemize}
\item If $\dim(X)=1$ then the statement has content only if $\deg K_X<0$.   Let $L=H^0(X,\sO_X)$.  Then  \cite[Lemma 6.5]{cascini_base_2015} says that $X\otimes_L\overline{K}$ is a conic in $\mathbb{P}^2_{\overline{K}}$, of which the reduction of a component is a rational curve.  Via the factorization $X\otimes_L\overline{K}\to X\otimes_{K}\overline{K}\to X$, we see that \(X\) is also the image of a rational curve on $X\otimes_K\overline{K}$. 

\item If $\dim(X)=2$ we must examine the proof of \cite[Theorem 4.3]{das_waldron_imperfect}. This proceeds by applying a known version of the cone theorem to isolate the extremal rays and then showing that each one contains a curve of the required form.  So we split into cases based on the type of contraction $f:X\to Z$ of the extremal ray.  If $f$ is birational then the curve found at the very end of \cite[Theorem 4.3]{das_waldron_imperfect} satisfies $\deg K_{C_i}<0$ and hence is the image of a rational curve by the dimension $1$ case.  If $f$ has relative dimension one then the curve is a general fibre of this morphism which again has  $\deg K_{C_i}<0$ by adjunction.  Finally if $f$ has relative dimension two, the curve is found by applying bend and break \cite[II.5.8]{kollar_rational_1996} on $Y$ and taking the image on $X$, and therefore the curve found on $Y$ is rational. 
\item If $\dim(X)=3$ then we find the curves in the proof of \cite[Lemma 7.1]{das_waldron_imperfect}.  There are three cases considered, the first of which is finding a curve inside a divisorial extremal ray.  In this case we cite the argument of \cite[Lemma 3.2]{birkar_existence_2017}, which restricts to and produces a Mori fibre structure on the divisor contracted by the ray.  Hence the curves found come from rational curves by the $\dim(X)=2$ case.
The next case is to deal with a flipping contraction, which again cites \cite[Lemma 3.2]{birkar_existence_2017}, where this case again concludes by applying the cone theorem for surfaces. 
Finally we need to find curves in a contraction of positive relative dimension.  If the relative dimension is not three, then again \cite{das_waldron_imperfect} cites the argument of \cite[Lemma 3.3]{birkar_existence_2017} which proceeds by applying the cone theorem on the restriction to the pullback of a general hyperplane, so again the rationality comes from the surface case.  Finally if the relative dimension is $3$ then the curve comes from applying bend and break on $Y$, which produces rational curves as in the surface case.
\end{itemize}

\end{proof}

We also proved the base point free theorem for big divisors, which together with the cone theorem allows us to contract birational extremal rays:

\begin{theorem}\cite[Theorem 1.4]{das_waldron_imperfect}\label{thm:big_bpf}
Let $(X,B)$ be a projective klt threefold pair over a field of characteristic $p>5$ with $\mathbb{Q}$-boundary.  Then if $L$ is a $\mathbb{Q}$-divisor such that both $L$ and $L-(K_X+B)$ are nef and big, then $L$ is semi-ample.
\end{theorem}

\begin{theorem}\cite[Theorem 1.6]{das_waldron_imperfect}\label{thm:lmm_dw}
Let $(X,B)$ be a quasi-projective klt threefold pair over an $F$-finite field $k$ of characteristic $p>5$.  Then if $K_X+B$ is pseudo-effective, $(X,B)$ has a log minimal model.
\end{theorem}

\begin{remark}
We will prove this result in the non-$F$-finite case in \autoref{sec:non_F-finite}.
\end{remark}

\section{Extremal rays and Tsen's theorem}

Tsen's theorem provided a key step in the construction of Mori fibre spaces for threefolds over algebraically closed fields \cite[Lemma 3.5]{birkar_existence_2017} by allowing us to control the behaviour of polytopes of boundary divisors.  While Tsen's theorem does not hold in our situation, in this section we get around this issue to recover the results on boundedness of extremal rays \cite[Proposition 3.8]{birkar_existence_2017}.

\begin{lemma}\label{lem:tsen}
Let $(X,B)$ be a $\mathbb{Q}$-factorial dlt pair of dimension $3$ over a field $K$ of characteristic $p>5$.  Assume that $f:X\to Z$ is a $(K_X+B)$-negative birational extremal contraction such that $-S$ is ample over $Z$ for some component $S$ of $\rddown{B}$.  Let $C$ be a curve on $Z$.  Then there is a curve $D$ on $X$  whose image on $Z$ is $C$ and  such that
\[ \frac{L\cdot C}{d_C}=\frac{f^*L\cdot D}{d_D}\]
for any Cartier divisor $L$ on $Z$.
\end{lemma}
\begin{proof}
By perturbing we may assume that $B$ is a $\mathbb{Q}$-divisor.  If $C$ is not contained in the image of the exceptional locus of $f$, then the statement is clear, since we can let $D$ be the strict transform of $C$ on $X$, and $D\to C$ is an isomorphism over an open subset of $C$.  Then we have $f_*D=C$ and so the projection formula gives $L\cdot C=f^*L\cdot D$.  
Similarly we see that $d_C=d_D$ because we can choose $D^{\overline{K}}$ and $C^{\overline{K}}$ such that $f_{\overline{K},*}D^{\overline{K}}=C^{\overline{K}}$ since $D^{\overline{K}}\to C^{\overline{K}}$ is an 
isomorphism over an open subset of $C^{\overline{K}}$ when we compute it via the formula in \autoref{thm:cone}.  In particular, the same argument implies that
$$f_{\overline{K}}^*\phi_Z^*L\cdot_{\overline{K}} D^{\overline{K}}=\phi_Z^*L\cdot_{\overline{K}}C^{\overline{K}}$$ where $\phi_Z:Z\otimes_K\overline{K}\to Z$ are the base change to $\overline{K}$.

So we only need to show the case that $f$ is a divisorial contraction and $C$ is the image of the contracted divisor, which must be $S$.  Note that $S$ is normal by \cite[Corollary 7.17]{bhatt_globally}.
By adjunction we may write $K_S+B_S\sim_{\mathbb{Q}}(K_X+B)|_S$, where $(S,B_S)$ is dlt, and since this is a $(K_X+B)$-negative extremal contraction, we have that $-(K_S+B_S)$ is ample over $C$.  
 
We next claim that $f_*\mathcal{O}_S=\mathcal{O}_C$. If we pushforward the exact sequence \[0\to \mathcal{O}_X(-S)\to \mathcal{O}_X\to \mathcal{O}_S\to 0\] by $f$ we obtain a long exact sequence
$$0\to \mathcal{I}_C\to \mathcal{O}_Z\to f_*\mathcal{O}_S\to R^1f_*\mathcal{O}_X(-S)\to\cdots$$ where $\mathcal{I}_C$ is the ideal sheaf of $C$ in $Z$.  
  We have \(-S\sim K_X+B-S-(K_X+B)\) is ample over $Z$, so by 
  \cite[Proposition 2.3]{bernasconi_kollar}, it suffices to prove that Kawamata-Viehweg holds for \(f|_S\).  Since \(S\to C\) has relative dimension \(1\) this holds by \cite[Proposition 3.2]{tanaka_minimal_2016-1}. 

Let $S^{\overline{K}}$ and $C^{\overline{K}}$ be components of $(S\otimes_K\overline{K})_{\mathrm{red}}$ and $(C\otimes_K\overline{K})_{\mathrm{red}}$ respectively, chosen in such a way that we have a morphism
 $f^{\overline{K}}:S^{\overline{K}}\to C^{\overline{K}}$ induced by the base change of $f:S\to C$ to $\overline{K}$. 
Since formation of the Stein factorisation commute with flat base change and taking reduced subschemes, we see that $f^{\overline{K}}$ also satisfies $f^{\overline{K}}_*\mathcal{O}_{S^{\overline{K}}}=\mathcal{O}_{C^{\overline{K}}}$. 
  Note that by \autoref{thm:cone} we have that $\frac{L\cdot C}{d_C}=\psi^*L\cdot C^{\overline{K}}$ where $\psi:C^{\overline{K}}\to C$. 
Let $T$ be the normalisation of $S^{\overline{K}}$ with morphism $\phi:T\to S$.  Then by \cite[Theorem 1.1]{patakfalvi_singularities_2017} there is an effective Weil divisor $\Delta$ such that 
\[\phi^*(K_S+B_S)=K_T+(p-1)\Delta+B_T\]
where $B_T=\phi^*B_S$.
By \cite[Theorem 1.1]{ji_waldron}, the support of $\Delta$ contains the entire locus on which $T\to S^{\overline{K}}$ fails to be an isomorphism.  Since the generic fibre $D$ of $T\to C^{\overline{K}}$ is a curve such that $\deg K_D<0$, we see that $\Delta$ must be vertical over $C^{\overline{K}}$ since $p>2$.  Therefore $T\to (C^{\overline{K}})^{\nu}$ is a contraction, where $\nu$ denotes normalisation.

  Then let $T'$ be the minimal resolution of $T$, with 
\[\pi^*(K_T+(p-1)\Delta+B_{T})\sim_{\mathbb{Q}}K_{T'}+(p-1)\Delta_{T'}+B_{T'}\]
where $\Delta_{T'}$ is the strict transform of $\Delta$.
Since $K_{T'}$ is not pseudo-effective over $C^{\overline{K}}$, running a $K_{T'}$-MMP over $C^{\overline{K}}$ results in a Mori fibre space, which is a $\mathbb{P}^1$-bundle over  $(C^{\overline{K}})^\nu$.  Since $p>2$, every component of $\Delta_{T'}$ is vertical for this Mori fibre space.  Now by Tsen's theorem there is a curve $D^{\overline{K}}$ on $T$ which maps birationally to  $(C^{\overline{K}})^\nu$, and hence for any Cartier divisor $L^{\overline{K}}$ on $C^{\overline{K}}$ we have \[L^{\overline{K}}\cdot C^{\overline{K}}=\theta^*f^{{\overline{K}}*}L^{\overline{K}}\cdot D^{\overline{K}}.\]
This is not quite enough: to deduce the required formula, we need to know that the normalisation $\theta: T\to (S^{\overline{K}})_{\red}$ is an isomorphism at the generic point of $D^{\overline{K}}$.  However, this follows by the fact noted earlier, that every component of the conductor of $T\to S^{\overline{K}}$ is a component of $\Delta$, which is vertical over $C^{\overline{K}}$.
\end{proof}

\begin{lemma}\label{lem:dominating_curve}
Let $(X,B)$ be a klt pair of dimension $3$ over a  field $K$ of characteristic $p>5$, and $C$ a curve on $X$.  Let $\phi:W\to X$ be a log resolution of $(X,B)$.  Then there is a curve $D$ on $W$ such that the image of $D$ in $ X$ is $C$, and we have
\[ \frac{L\cdot C}{d_C}=\frac{\phi^*L\cdot D}{d_D}\]
for any Cartier divisor $L$ on $X$.
\end{lemma}
\begin{proof}
This is now identical to the proof of \cite[Lemma 3.5]{birkar_existence_2017} except that we use \autoref{lem:tsen} in place of \cite[Lemma 3.6]{birkar_existence_2017}, and modify the statement accordingly.
\end{proof}

\begin{definition}\label{def:shokurov_polytope}\cite{shokurov_three-dimensional_1992}
Let $X$ be a klt, $\mathbb{Q}$-factorial threefold projective over a field $K$ of characteristic $p>5$.  Let $B\geq 0$ be a $\mathbb{Q}$-divisor on $X$ and 
$V$ be a finite dimensional rational affine space of the vector space of $\mathbb{R}$-Weil divisors.  Then we define the \emph{Shokurov polytope}
\[\mathcal{L}_B(V)=\{\Delta\mid 0\leq (\Delta-B)\in V, \ (X,\Delta)\mathrm{\ is\ lc}\}.\]
This is a rational polytope in $V$, since log resolutions exist in our situation.
\end{definition}

\begin{proposition}\cite[Proposition 3.8]{birkar_existence_2017}\label{prop:polytopes}
Let $X$ be a $\mathbb{Q}$-factorial projective klt threefold over a field $K$.  Let $V$ be a finite dimensional rational affine space of Weil divisors, and let $\mathcal{L}$ be a rational polytope inside $\mathcal{L}_A(V)$ for some $A$, and fix $B\in\mathcal{L}$.  Then there are real numbers $\alpha,\delta$ depending only on $(X,B)$ and $V$ such that 
\begin{enumerate}
\item if $\Gamma$ is any extremal curve and $(K_X+B)\cdot\frac{\Gamma}{d_{\Gamma}}>0$ then $(K_X+B)\cdot\frac{\Gamma}{d_{\Gamma}}>\alpha$.
\item  If $\Delta\in \mathcal{L}$, $||\Delta-B||<\delta$ and $(K_X+\Delta)\cdot R\leq 0$ for an extremal ray $R$ then \[(K_X+B)\cdot R\leq 0.\] 
\item\label{itm:polytope_polytope} Let $\{R_t\}_{t\in T}$ be a family of extremal rays of $\overline{NE}(X)$.  Then the set 
\[\mathcal{N}_T=\{\Delta\in\mathcal{L}\mid (K_X+\Delta)\cdot R_t\geq 0\mathrm{\ for\ all\ }t\in T\}\]
is a rational polytope.
\item Assume that $K_X+B$ is nef, $\Delta\in\mathcal{L}$ and that $X_i\dashrightarrow X_{i+1}/Z_i$ is a sequence of $(K_X+\Delta)$-flips which are $(K_X+B)$-trivial and $X=X_1$.  Then for any curve $\Gamma$ on any $X_i$ we have $(K_{X_i}+B_i)\cdot\frac{\Gamma}{d_\Gamma}>\alpha$ if $(K_{X_i}+B_i)\cdot \frac{\Gamma}{d_\Gamma}>0$ where $B_i$ is the birational transform of $B$. 
\item\label{itm:polytope_flips} In addition to the assumptions of $(4)$, suppose that $||\Delta-B||<\delta$.  If $(K_{X_i}+\Delta_i)\cdot R\leq 0$ for an extremal ray $R$ on some $X_i$ then $(K_{X_i}+B_i)\cdot R=0$ where $\Delta_i$ is the birational transform of $\Delta$. 
\end{enumerate}

\end{proposition}
\begin{proof}
The first three parts were already dealt with in \cite[Proposition 9.32]{bhatt_globally}, where we simply remarked that the proof from \cite{birkar_existence_2017} goes through after replacing all occurrences of curves $\Gamma$ with $\frac{\Gamma}{d_{\Gamma}}$.  As also remarked there, the proofs of the last two parts do not go through immediately, but \autoref{lem:dominating_curve} is exactly what we need to fix them.  In particular when one replaces curves $D$ and $\Gamma$ with $\frac{D}{d_D}$ and $\frac{\Gamma}{d_{\Gamma}}$ in the long displayed equation in the proof of part (4) in \cite[Proposition 3.8]{birkar_existence_2017}, the formula given by \autoref{lem:dominating_curve} is exactly what is needed to replace the condition ``$D\to\Gamma$ is birational''.
\end{proof}

\section{Nef reduction maps}\label{sec:nef_red}

The nef reduction map will allow us to obtain a first approximation of the semi-ample fibration.  It was first introduced over algebraically closed fields of characterstic zero in \cite{bauer_reduction_2002}  and in positive characteristic in \cite{cascini_base_2015}.  In this section we will introduce it for arbitrary ground fields.

\begin{lemma}\label{lem:descend_rational_map}
Let $\pi:Y\to X$ be a finite purely inseparable morphism of quasi-projective varieties, and $f:Y\dashrightarrow V$ a rational map.  Then there is a rational map $g:X\dashrightarrow Z$ and purely inseparable morphism $\phi:V\to Z$ making the diagram commute:
\[
\begin{tikzcd}
Y\ar[r,"\pi"]\ar[d,dashrightarrow,"f"] & X\ar[d,dashrightarrow,"g"]\\
V\ar[r,"\phi"] & Z
\end{tikzcd}
\]

\end{lemma}
\begin{proof}

By taking the normalisation of the closure of the graph of $f$, we may replace $Y$ and $X$ to assume that $f$ is a morphism.  Since $\pi$ is finite and purely inseparable there is a purely inseparable morphism $\alpha:X\to Y$ such that $\pi\circ\alpha=F_X^n$ is the absolute Frobenius of $X$.  Then let $g:X\to Z$ be the Stein factorisation of $f\circ \alpha$, with morphism $\beta:Z\to V$.  Then let $\phi$ be the purely inseparable morphism such that $\phi\circ\beta=F_Z^n$.  Then we have a commutative diagram: 
\[
\begin{tikzcd}
X\ar[r,"\alpha",swap]\ar[d,"g"]\ar[rr,"F_X^n", bend left=20] & Y\ar[r,"\pi",swap]\ar[d,"f"] & X\ar[d,"g"]\\
Z\ar[r,"\beta"]\ar[rr,"F_Z^n", bend right=20, swap] &V\ar[r,"\phi"] & Z
\end{tikzcd}
\]

The same map $g$ fits into the right hand side of the diagram by functoriality of Frobenius.
\end{proof}

\begin{proposition}\label{prop:nef_red_exists}
Let $X$ be a normal projective variety over an uncountable field $K$ with $H^0(X,\sO_X)=K$, and let $D$ a nef $\mathbb{R}$-divisor on $X$.  Then there is a separable extension $L/K$, a rational map $f:X_L\dashrightarrow Z$ to a projective variety $Z$, and an open subset $V\subset Z$ such that
\begin{enumerate}
\item\label{itm:nef_red_proper} $f$ is proper over $V$,
\item\label{itm:nef_red_triv} $D|_F\num 0$ for the very general fibres $F$ of $f$ over $V$,
\item\label{itm:nef_red_exact} if $x\in X_L$ is a very general point and $C$ a curve passing through $x$ then $D\cdot C=0$ if and only if $C$ is contained in the fibre of $f$ containing $x$.
\end{enumerate}
Furthermore, since $f$ is proper over $V$, we can replace $Z$ with the Stein factorisation to assume that $f_*\sO_{f^{-1}(V)}=\sO_V$.
\end{proposition}

\begin{proof}
We first construct the map.  Let $Y$ be the normalisation of $X\otimes_K \overline{K}$, with morphism $\pi:Y\to X$, and let $D_Y=\pi^*D$.  Nefness is preserved by pullback by the projection formula, and so  $D_Y$ has a nef reduction map $g:Y\dashrightarrow Z_{\overline{K}}$ by \cite[Theorem 2.9]{cascini_base_2015}.  Let $E$ be a finite extension of $K$ such that $Y$, $Z_{\overline{K}}$ and $g$ are defined over $E$.   More precisely there is a rational map of varieties over $E$, $g_E:Y_E\dashrightarrow Z_E$ which base changes to $g$.  Let $L$ be the unique subfield of $E$ such that $E/L$ is purely inseparable and $L/K$ is separable.  It is sufficient to show that we can descend $g$ via $E/L$, but this follows from the more general \autoref{lem:descend_rational_map}.  In particular, use \autoref{lem:descend_rational_map} to fill in the diagram:
\[
\begin{tikzcd}
Y_E\ar[r]\ar[d,dashrightarrow]  &X_L\ar[d,dashrightarrow]\\
Z_E\ar[r]& Z
\end{tikzcd}
\]

We now have to show that the rational map defined this way has the required properties.  Property \autoref{itm:nef_red_proper} is clear.  
Suppose that $f$ is a very general fibre of $F$ over a point $v\in V$.  The preimage of $v$ in $Z_{\overline{K}}$ is a finite set of points, and we can ensure that these are all very general by choosing $v$ such that it avoids the image of the given countable union of proper subvarieties of $Z_{\overline{K}}$.
Hence for a curve $C$ in $F$, every component of $C\otimes_L \overline{K}$ lie in a very general fibre of $g$. This implies that condition \autoref{itm:nef_red_triv} is satisfied, by applying the projection formula and property \autoref{itm:nef_red_triv} over algebraically closed fields.

Finally, a very general point $x\in X_L$ lies under only very general points of $Y$ by choosing it in such a way that it avoids the images of the countably many bad subvarieties.  Property \autoref{itm:nef_red_exact} follows from the projection formula, by noting that for any curve $C$ through $x$, if $\widetilde{C}$ is a component of $(C\otimes_K L)_{\red}$ then $C\cdot D_L=0$ if and only if $\widetilde{C}\cdot D_{\overline{K}}=0$.

\end{proof}

\begin{definition}
We call the rational map constructed in \autoref{prop:nef_red_exists} the \emph{nef reduction map} of $D$. $\dim(Z)$ is called the \emph{nef dimension} of $D$, and denoted $n(D)$.  

Note that we always have $n(D)\geq \kappa(D)$ by \cite[Proposition 2.8]{bauer_reduction_2002}, and the fact that $\kappa(D)$ is preserved by pullbacks.
\end{definition}

\begin{remark}
It might be that the separable base change is not actually needed.
\end{remark}

\begin{proposition}\label{prop:max_nef_dim}
Let $X$ be a projective variety over an uncountable field $K$, and $A$ and $B$ be effective $\mathbb{R}$-divisors such that $A$ is nef and big and $K_X+A+B$ is nef.  

If $n(K_X+A+B)=\dim(X)$ then $K_X+A+B$ is big.
\end{proposition}
\begin{proof}
Let $Y$ be the normalisation of $X\otimes_K \overline{K}$ with morphism $\pi:Y\to  X$.  Then \[\phi^*(K_X+A+B)=K_Y+B_Y+\phi^*A\] for some  $B_Y\geq 0$ by \cite[Theorem 1.1]{tanaka_behavior_2016}, and \[n(K_X+A+B)=n(K_Y+B_Y+\phi^*A_Y)\] by definition of the nef reduction map.  Hence the statement follows from the case of algebraically closed fields \cite[Theorem 1.11]{birkar_existence_2016}, noting that this did not require the divisors to be boundaries.
\end{proof}

The following will be used later to ensure that a $D$-trivial MMP does not change the nef reduction map.

\begin{definition}
Let $f:X\dashrightarrow Z$ be the nef reduction map of $D$.  Since $f$ is proper over an open subset $U$ of $Z$, we can say that an integral divisor $E$ on $X$ is \emph{horizontal} over $Z$ if $E\cap f^{-1}(U)\to U$ is surjective, and \emph{vertical} otherwise.
\end{definition} 

\begin{proposition}\label{prop:same_nef_red_map}
Let $X$ be a $\mathbb{Q}$-factorial projective variety with nef $\mathbb{Q}$-divisor $D$, for which $f:X\dashrightarrow Z$ is the nef reduction map.  Suppose that $(X,B)$ is a
klt pair, and $\pi:X\to Y$ a $(K_X+B)$-negative birational extremal contraction such that $D\equiv_\pi 0$.
Then every component of the exceptional locus of $\pi$ is vertical over $Z$.
\end{proposition}

\begin{proof}
Note that since $(X,B)$ is klt and $\pi$ is $(K_X+B)$-negative and $D$-trivial, $D\sim_{\mathbb{Q},\pi}0$ by \autoref{thm:big_bpf}.  Therefore there exists a nef $\mathbb{Q}$-Cartier divisor $D_Y$ on $Y$ such that $D\sim_{\mathbb{Q}}\pi^*D_Y$.  Let $g:Y\dashrightarrow V$ be the nef reduction map of $D_Y$.  We may replace $g$ and $f$ by their Stein factorisations to assume that their general fibres are irreducible.  

When we take a very general point $x\in X$, we can simultaneously assume that: 
\begin{enumerate}
\item Condition (3) of \autoref{prop:nef_red_exists} applies at $x$.
\item $x$ lies in $f^{-1}(U)$, so that $f$ is proper near $x$.
\item $\pi$ is an isomorphism near $x$, and $\pi(x)$ is a very general point of $Y$ for the purposes of Condition (3) of \autoref{prop:nef_red_exists} applied to $g$.
\item $g$ is proper near $\pi(x)$.
\item $x$ lies in an irreducible fibre of $f$, and $\pi(x)$ lies in an irreducible fibre of $g$. 
\end{enumerate}
We may assume each of these simultaneously since each independently holds away from countably many proper closed subvarieties of $X$ or equivalently of $Y$. 

For such a very general point $x$, let $F_x$ be the  fibre of $f$ which contains $x$, and $G_y$ be the fibre of $g$ through $y=\pi(x)$.
Then a curve $C$ through $x$ is contained in $F_x$ if and only if $C\cdot D=0$, which happens if and only if $\pi_*C\cdot D_Y=0$ by the projection formula.  This in turn happens if and only if $\pi_*C$ is contained in  $G_y$.  As a result, we see that $G_y$ is the strict transform  $F_x$ since they contain the same curves when restricted to the locus on which $\pi$ is an isomorphism.

Now suppose for contradiction that the exceptional locus $E\subset X$ of $\pi$ is horizontal over $U$.  Let $\pi|_E:E\to W\subset Y$ be the induced contraction.  Since $E\to Z$ is surjective, there must be some point $w\in W$, such that $\pi^{-1}(w)$ contains infinitely many distinct points $e_i\in E$ such that $f(e_i)$ is very general in $Z$ and each fibre over $f(e_i)$ contains a very general point $x_i\in X$.  By very generality, $G_{\pi(x_i)}$ is an irreducible fibre of $g$ for each $i$.   But also $G_{\pi(x_i)}$ is the strict transform of $F_{x_i}$ which each contain $e_i$.  Hence $G_{\pi(x_i)}$ each contain $w$, so cannot be disjoint fibres of $g$.  This gives the required contradiction.

\end{proof}

\section{ACC}\label{sec:acc}

We now prove the global ACC for surfaces over imperfect fields.  The key idea is to run induction on the transcendence degree of the ground field.  Note that even though we care only about projective surfaces over fields, we crucially use the full power of the relative LMMP for three dimensional excellent schemes developed in \cite{bhatt_globally}.  

\begin{lemma}\label{lem:geom_reduced}
	Let $(X,B)$ be a dlt projective surface pair over a field of characteristic $p>5$ such that $B$ is big and $K_X+B\num 0$.  Then $X$ is geometrically reduced over $k:=H^0(X,\sO_X)$.  
\end{lemma}
\begin{proof}
	First note that since geometric reducedness of normal varieties is determined at the generic point, it is enough to show that some birational model of $X$ is geometrically reduced over $k$.  
	Let $\phi:Y\to X$ be the minimal resolution of $X$, and define $B_Y$ by $\phi^*(K_X+B)=K_Y+B_Y$.  Note that the assumptions of the theorem also apply to the pair $(Y,B_Y)$.   Let $\psi:Y\to W$ be the result of applying a $K_Y$-MMP, so $W$ is regular and carries a Mori fiber space structure $g:W\to Z$.  Since $K_Y+B_Y\num 0$, $K_Y+B_Y\sim_{\mathbb{Q}} \psi^*(K_W+B_W)$ for a log canonical pair $(W,B_W)$.    If $\dim(Z)=0$ then $W$ is geometrically reduced by \cite[Corollary 1.4]{patakfalvi_singularities_2017}.  Hence we may assume that $\dim(Z)=1$. 
	
	If $\rddown{B_Y}=0$ then $X$ is log Fano type, since we may write $B_Y\sim_{\mathbb{Q}}A+E$ for some effective $E$ and ample $A$.  Then for $\epsilon$ sufficiently small, $(Y, (1-\epsilon B_Y)+\epsilon E)$  is klt and $-(K_Y+(1-\epsilon)B_Y+\epsilon E)$ is ample.  
	Therefore $X$ is geometrically reduced by \cite[Corollary 5.5]{bernasconi_tanaka}, and so we may assume that $\rddown{B_Y}\neq 0$.  
	
	First suppose there is some component $E$ of $\rddown{B_Y}$, and hence of $\rddown{B_W}$, which is dominant over $Z$.  Since we have $(K_W+B_W)|_E\sim_{\mathbb{Q}}K_E+B_E$ for some $B_E\geq0$, we have that $E$ is smooth over $H^0(E,\sO_E)$ since $p\geq 5$.  Furthermore, since the generic fiber of $g$ is a rational curve, the degree of $E\to Z$ is at most two and $Z$ is also smooth over $k = H^0(Z,\sO_Z)$.  Therefore smoothness of  $W$ follows from \cite[Proposition 2.18]{bernasconi_tanaka}.
	
	Therefore we can assume that no component of $\rddown{B_Y}$ is horizontal over $Z$.  
	For $z\in Z$, denote the fibre of $Y\to Z$ over $z$ by $F_z$.   Then let $I$ be the set of all $z\in Z$ for which there is a component of $\rddown{B_Y}$ contained in $F_z$. Denote $\sum_{z\in I} F_z=E+D$, where the  support of $E$ is contained in $\rddown{B_Y}$ and no component of $D$ is contained in $\rddown{B_Y}$.  By performing blow-ups of zero dimensional strata of $\rddown{B_Y}$, we may assume that no component of $D$ passes through a zero-dimensional strata of $\rddown{B_Y}$.  Since the strata of $\rddown{B_Y}$ are all the non-klt centers of $(Y,B_Y)$, it follows that $(Y,B_Y+\epsilon D)$ is dlt for $\epsilon$ sufficiently small.  But then for a general fiber $F'$ and $\delta\ll \epsilon $ we have $(Y,B_Y':=B_Y-\delta E+\delta D+\delta F')$ klt. But this boundary is still big, and so by the same argument as above, $Y$ is log Fano type and so we can  again apply \cite[Corollary 5.5]{bernasconi_tanaka} and are done.  
\end{proof}

\begin{proposition}\label{prop:relative_ACC}
Let $\Lambda\subset [0,1]$ be a DCC set of real numbers.  Then there is a finite subset $I\subset\Lambda$ with the following property:

Let $(X,B)$ be a  pair over a perfect field $k$ of characteristic $p>5$, and $f:X\to Z$ be a projective morphism of normal quasi-projective varieties such that
\begin{itemize}
\item $(X,B)$ is dlt
\item $B$ is $f$-big
\item $f$ has relative dimension  $2$.
\item The coefficients of $B$ are in $\Lambda$ 
\item $K_X+B\equiv_f 0$ 
\end{itemize}
Then the coefficient of each $f$-horizontal component of $B$ is in $I$.
\end{proposition}

\begin{proof}
\emph{Step 1: Easy reductions}

Since the statement is local at the generic fibre of $f$, we are free to shrink $Z$ as required.  We may also replace $Z$ by the Stein factorisation of $f$ to assume that $f_*\sO_X=\sO_Z$.    Denote the function field of $Z$ by $K$.
Let $\xi$ be the generic point of $Z$, and $X_{\xi}$ the generic fibre of $f$, which is a projective variety of dimension $2$.  $X_{\xi}$ is geometrically reduced by \autoref{lem:geom_reduced}.

\emph{Step 2: Preparing for induction.}

We do induction on $\dim(Z)$.  First suppose $\dim(Z)=0$, i.e. $X$ is a surface which is projective over a perfect field $k$.  If 
$k$ is algebraically closed, this is dealt with in \cite[Proposition 11.7]{birkar_existence_2016}. Otherwise performing a base change from a perfect field $k$ to $\overline{k}$ preserves the assumptions, and is enough to conclude the result.

So we may assume that $\dim(Z)>0$, and that the Proposition holds whenever the base has dimension $\dim(Z)-1$.   Let $H$ be a general hyperplane on $Z$ and let $F=f^*H$.  
Step 2 showed that the  fibres of $f$ over a non-empty open subset of $Z$ are integral.  Hence we can choose $H$ in such a way that $F=f^*H$ is integral over the generic point of $H$.  Furthermore if it is not integral everywhere, it must fail over some proper closed subset of $H$, so we can assume that $F$ is integral by shrinking $Z$.

Let $\eta$ be the generic point of $H$, and consider the localisation $Z_\eta$ which is a spectrum of a DVR, and the base change $X_\eta:=X\otimes_Z Z_{\eta}$.  Since $X_{\eta}$ is of dimension three and projective  over the excellent base $Z_{\eta}$, we may apply LMMP techniques to it using \cite{bhatt_globally}.  Note that the generic fibre of $X_\eta\to Z_{\eta}$ is isomorphic to the generic fibre of $f$.

Now let $\pi:(Y_\eta,\Gamma)\to X_{\eta}$ be a $\mathbb{Q}$-factorial dlt model of $(X_{\eta},B_{\eta}+F_{\eta})$.  To be more exact, it is the model obtained in \cite[Corollary 9.21]{bhatt_globally}, which in our situation has the following properties:
\begin{itemize}
\item  $\Gamma$ is the sum of the strict transform of $B_{\eta}+F_{\eta}$, and the reduced exceptional divisor of $\pi$, by definition
\item $(Y_{\eta},\Gamma)$ is a dlt pair.
\item  $K_{Y_{\eta}}+\Gamma$ is nef over $X$. 
\item  $\pi^{-1}(F_{\eta})$, the special fibre of $Y_\eta\to Z_\eta$, is contained in $\rddown{\Gamma}$, since $\Supp(\pi^{-1}(F_{\eta}))=\Supp(\pi^*f^*H_{\eta})$ is a Cartier divisor, hence any irreducible component of it is either a component of the strict transform of $F$, or an exceptional divisor.
\item $(K_{Y_{\eta}}+\Gamma)|_{Y_{\eta}\setminus \pi^{-1}(F_{\eta})}\equiv_f 0$,
 because the non-lc locus of $(X,B+F)$ is contained entirely in $\Supp(F)$, and so $K_{Y_{\eta}}+\Gamma$ is just the crepant pullback of $K_X+B$ away from $F$.
\end{itemize}

\emph{Step 3: Reducing to the case where $K_{Y_{\eta}}+\Gamma$ is semi-ample over $Z_{\eta}
$.}

Now we claim that we may run a terminating $(K_{Y_{\eta}}+\Gamma)$-MMP over $Z_\eta$.  
Since $K_{Y_{\eta}}+\Gamma$ is numerically trivial over the generic point of $Z_{\eta}$, it is $f$-pseudo-effective.  Therefore there is a $(K_{Y_{\eta}}+\Gamma)$-MMP over $Z_\eta$ which terminates by \cite[Proposition 9.20]{bhatt_globally}.  Furthermore, the MMP only contracts rays in the special fibre, since $K_{Y_\eta}+\Gamma$ is already numerically trivial on the generic fibre. 
 Hence the MMP terminates with a log minimal model $g:(W_{\eta},\Gamma_{W_\eta})\to Z_{\eta}$ whose generic fibre is isomorphic to that of $(Y_{\eta},\Gamma_{Y_\eta})\to Z_{\eta}$.  We now claim that $K_{W_{\eta}}+\Gamma_{W_\eta}$ is in fact $g$-semi-ample.  By \autoref{thm:cascini_tanaka}, it is sufficient to show that its restriction to every fibre of $g$ is semi-ample.  But since $Z_{\eta}$ is a DVR, there are exactly two fibres to check.  The restriction to the generic fibre is semi-ample by abundance for surfaces over imperfect fields \cite{tanaka_abundance_2015}.  For the special fibre, note that the fibre over $\eta$ is entirely contained in the support of $\rddown{\Gamma_{W_{\eta}}}$.  Hence we may apply adjunction to the fibre and abundance for semi-log canonical surfaces \cite{posva_abundance}.  Let $W_\eta\to V_\eta$ be the projective contraction over $Z_{\eta}$ associated to the semi-ample $K_{W_\eta}+\Gamma_{W_\eta}$.  Since $K_{W_{\eta}}+\Gamma_{W_\eta}$ was numerically trivial on the generic fibre and $f_*\sO_{W_\eta}=\sO_{Z_\eta}$, $V_\eta\to Z_{\eta}$ is birational, and in fact an isomorphism since $Z_{\eta}$ is regular and of dimension one.

\emph{Step 4: Spreading out over $Z$}

We have obtained a dlt pair $(W_{\eta},\Gamma_{W_{\eta}})$ such that $K_{W_{\eta}}+\Gamma_{W_{\eta}}$ is $\mathbb{Q}$-linearly trivial over $Z_{\eta}$.  
We now claim that we can spread this out over some open subset of the original $Z$ to obtain a corresponding projective morphism $W\to Z$ and dlt pair $(W,\Gamma_W)$ such that $K_W+\Gamma_W\sim_{\mathbb{Q},Z}0$.  
We can always spread out to form a pair $(W,\Gamma_W)$ which is projective over some open subset of $Z$, and 
since we know that the locus over which $K_W+\Gamma_W$ is $f$-semi-ample is open, we may shrink $Z$ without removing the generic point of $H$ in a way which ensures that $K_W+\Gamma_W$ is $f$-semi-ample.

Furthermore, if we take a log resolution $V_\eta\to W_\eta$ which verifies that $(W_\eta,\Gamma_{W_{\eta}})$ is dlt using \cite[Proposition 2.40]{kollar_birational_1998}, we can spread out $V_\eta$ over some open neighbourhood of $\eta$ in such a way that it remains a log resolution.
Furthermore by further shrinking $Z$ we can ensure that every exceptional divisor has image containing the generic point of $H$, and hence its log discrepancy is determined by that on $(W_\eta,\Gamma_{W_\eta})$.  This then verifies that $(W,\Gamma_W)$ is dlt after sufficiently shrinking $Z$.  Note that the coefficients of $\Gamma_W$ are drawn from the set $\Lambda\cup\{1\}$, which remains DCC.

\emph{Step 5: Adjunction}

Pick a component $C$ of $B$ which is horizontal over $Z$, whose coefficient we wish to show comes from a finite set. Since over the generic point of $Z$, $(Y_{\eta},\Gamma_{Y_{\eta}})$ was just the crepant pullback of $(X,B)$, and furthermore $Y_{\eta}\to W_{\eta}$ is an isomorphism over the generic point of $Z$, we see that the  strict transform $C_W$ of $C$ on $W$ exists.  Furthermore, it has the same coefficient in $\Gamma_W$ as it had in $B$.  
 Since $C_W$ is horizontal, we may choose some component $S$ of $\rddown{\Gamma_W}$ which is surjective over $H$ and such that $C|_S$ is horizontal over $H$.  
Let $K_S+\Gamma_S=(K_W+\Gamma_W)|_S\sim_{\mathbb{Q},H} 0$.

Comparing coefficients via adjunction \autoref{prop:adjunction_coefficients}
shows that the coefficient of each prime divisor $V\subset S$ in $\Gamma_S$ has the form 
\[a=\frac{l-1}{l}+\sum\frac{b_im_i}{l}+\frac{cn}{l}\]
where $l\in\mathbb{N}$, $m_i,n\in\mathbb{N}\cup \{0\}$, $c$ is the coefficient of $C$ in $\Gamma_W$, and $b_i'$ are the coefficients of the other components of $\Gamma_W$.
By the induction hypothesis, only finitely many of the numbers of this form can actually occur.  Therefore there can be only finitely many possibilities for $c$ which actually occur by \cite[Lemma 5.2]{hacon_acc_2014}.

\end{proof}

\begin{proof}[Proof of \autoref{thm:global_acc}]
We reduce the statement to \autoref{prop:relative_ACC} by spreading out.  Suppose that $(X,B)$ is a dlt pair over a field $K$ satisfying the requirements in \autoref{thm:global_acc}.  Furthermore, let $Z$ be the non-snc locus of $(X,B)$, and let $\pi:Y\to X$ be a log resolution such that every exceptional divisor over $Z$ has positive log discrepancy.  
Let $E$ be a transcendental extension of $\mathbb{F}_p$ or $\mathbb{Q}$ such that all equations of some quasi-projective embeddings of $X$, every component of $B$, $Z$, $Y$ and $\pi$ are all contained in $E$.  As a result there is a pair $(X_E, B_E)$ over $E$, and birational model $\pi:Y_E\to X_E$ which base change to the original pairs under field extension to $K$.  The set-up over $E$ satisfies all the same assumptions as the original one over $K$, except now it forms the generic fibre of a  projective morphism $f:\mathcal{X}\to \mathcal{V}$ where $E$ is the function field of $\mathcal{V}$ and with a pair $(\mathcal{X},\mathcal{B})$.  We claim that after shrinking $\mathcal{V}$ sufficiently this set-up satisfies the assumptions of \autoref{prop:relative_ACC}.  The only one which is not immediate is dlt.  However, since we also spread out a log resolution $\pi:\mathcal{Y}\to \mathcal{X}$, we may also shrink $\mathcal{V}$ to ensure that this forms a log resolution of $(\mathcal{X},\mathcal{B})$, and the exceptional divisors (and their log discrepancies) match those of $\pi:Y\to X$.  Hence $(\mathcal{X},\mathcal{B})$ is dlt, and so the theorem follows immediately from \autoref{prop:relative_ACC}  applied to this family if the characteristic of $K$ is positive and \cite[Theorem 1.5]{hacon_acc_2014} if the characteristic of $K$ is zero.
\end{proof}

Now the ACC for log canonical thresholds for threefold singularities with residue characteristic $p>5$:

\begin{proof}[Proof of \autoref{thm:acc_for_lct}]
Suppose we have a log canonical threefold pair $(X,B)$ and $\mathbb{R}$-Cartier divisor $M$, and let $\lambda=\mathrm{lct}(M,X,B)$.  We may assume that no component of $M$ is a log canonical center of $(X,B+\lambda M)$ since the coefficient of this component would be $1=b+\lambda m$ where $b$ is its coefficient in $B$ and $m$ is its coefficient in $M$.  Since $b$ and $m$ come from DCC sets, $\frac{1-b}{m}$ comes from an ACC set.  Therefore we may assume that there is a log canonical center $Z$ of dimension at most $1$.  Let $\pi:Y\to X$ be a $\mathbb{Q}$-factorial dlt model of $(X,B+\lambda M)$ with
\[\pi^*(K_X+B+\lambda M)=K_Y+B_Y+\lambda M_Y+E\]
where $B_Y$ and $M_Y$ are the strict transforms of $B$ and $M$ respectively, and $E$ is the reduced exceptional divisor.  Choose a component $F$ of $E$ which dominates $Z$, and such that some component of $M_Y$ is horizontal over $Z$.

Applying adjunction to $F$ gives 
$$K_{F}+\Delta_{F}=(K_{Y}+B_{Y}+\lambda M_Y+E)|_{F}$$
where a component of $\Delta_F$ has  coefficient of the form:

\[a=\frac{l-1}{l}+\sum\frac{b_ik_i}{l}+\sum\frac{\lambda m_in_i}{l}\]

where $l\in \mathbb{N}\cup\{\infty\}$, $k_i,n_i\in\mathbb{N}\cup\{0\}$, $b_i$ are the coefficients of $B$ and $m_i$ are the coefficients of $M$.  In particular, since $M$ is horizontal over $Z$, there is at least one component of $\Delta_F$ such that $n_i\neq 0$ for some $i$.  In particular, restricting to the generic point of $Z$ gives a pair satisfying the hypotheses of \autoref{thm:global_acc}, and hence the possible coefficients $a$ come from a finite set, either by \autoref{thm:global_acc} if $F$ has dimension 2 or by \cite[Proposition 5.4]{das_boundedness} if $F$  has dimension $1$. 
But this implies that only finitely many possibilities for $\lambda$ can occur by \cite[Lemma 5.2]{hacon_acc_2014}.
\end{proof}

\begin{proof}[Proof of \autoref{cor:lc_termination}]
The proof is as in \cite{birkar_acc_2005}, and is written in detail in \cite[Theorem 1.6]{waldron_lmmp_2017}.  The necessary inputs for the argument to work are ACC for log canonical thresholds (\autoref{thm:acc_for_lct}), special termination \cite[Theorem 6.6]{das_waldron_imperfect}, a crepant $\mathbb{Q}$-factorialisation \cite[Corollary 9.21]{bhatt_globally} and the ability to run a (not necessarily terminating) LMMP for a  $\mathbb{Q}$-factorial dlt pair \cite[Theorem 1.1, Theorem 1.3, Theorem 1.5]{das_waldron_imperfect} or \cite[Theorem 9.28, Theorem 9.33, Theorem 9.14]{bhatt_globally}. 
\end{proof}

\section{Goodness}\label{sec:goodness}

The aim of this section is to show that if $(X,B)$ is klt, $B$ is big and $K_X+B$ is nef, then the nef dimensions and Kodaira dimensions of $K_X+B$ are equal.
The following Lemma was originally proved by Kawamata \cite{kawamata_pluricanonical_1985}, and the  version in our current generality follows immediately from that in \cite[Lemma 9.25]{bhatt_globally}.

\begin{lemma}\label{lem:kawamata_lemma}\cite[Lemma 9.25]{bhatt_globally}
Let $f:X\to T$ be a projective and surjective contraction between normal varieties over a field $K$.  Let $L$ be a $\mathbb{Q}$-Cartier $\mathbb{Q}$-divisor on $X$ which is nef over $T$ and such that $L|_F$ is semi-ample, where $F$ is the generic fibre of $f$.  Assume that $\dim(X)\leq 3$.  Then there exists a commutative diagram 
\[
\begin{tikzcd}
X'\ar[r,"\phi"]\ar[d,"f'"] &X\ar[d,"f"]\\
Z\ar[r,"\psi"] & T\\
\end{tikzcd}
\]
such that $\phi$ and $\psi$ are projective birational, $f'$ agrees with the map induced by $\phi^*L$ over the generic point of $T$, and 
\[\phi^*L\sim_{\mathbb{Q}}f'^*D.\]
for some $\mathbb{Q}$-Cartier $\mathbb{Q}$-divisor $D$ on $Z$. 
\end{lemma}

\begin{proof}
We just note that the proof of \cite[Lemma 9.25]{bhatt_globally} provided both morphisms being birational as claimed.
\end{proof}

We use this to show that if $B$ is big and $K_X+B$ is nef, then it always pulls back from the base of its nef reduction map.  This is a first approximation to semi-ampleness.

\begin{proposition}\label{prop:pullback_from_nef_red}
Let $(X,B)$ be a $\mathbb{Q}$-factorial klt pair of dimension $3$ over an uncountable field $K=H^0(X,\sO_X)$.  Suppose that $B$ is a big $\mathbb{Q}$-divisor and $K_X+B$ is nef.  Then there exists a separable extension $L/K$, a projective birational contraction $\phi:W\to X_L$ and projective contraction $h:W\to Z$ with $W$ and $Z$ normal, such that
\begin{enumerate}
\item $\dim(Z)=n(K_X+B)$
\item There is a nef divisor $D$ on $Z$ such that  \[\phi^*(K_X+B)\sim_{\mathbb{Q}}h^*D.\]
\end{enumerate}
\end{proposition}

\begin{proof}
Let $f:X_L\dashrightarrow Z$ be the nef reduction map of $K_X+B$, where $L$ is a finite separable extension of $K$.  By replacing $X$ by $X_L$ and performing a small $\mathbb{Q}$-factorialisation \cite[Corollary 9.21]{bhatt_globally}, we may assume that $L=K$ from now on.
Let $W$ be the normalisation of the graph of $f$, with morphisms  $f_W:W\to Z$ and $\phi:W\to X$.

 If $n(K_X+B)=3$ the theorem is trivial by taking $X=Z=W$.  

If $n(K_X+B)=1$ or $2$, let $F$ be the generic fibre of $f_W:W\to Z$, which since $f$ is proper over an open subset of $Z$, is equal to the generic fibre of $f$. We have 
\[0\equiv (K_X+B_X)|_F=K_F+B_F.\] 
$(F,B_F)$ is a klt pair of dimension $1$ or $2$, with $K_F+B_F$ numerically trivial.  Hence $K_F+B_F$ is semiample by \cite{tanaka_abundance_2015}, and we can conclude via \autoref{lem:kawamata_lemma}.

Finally we are left to deal with the case of $n(K_X+B)=0$, which is equivalent to $K_X+B\equiv 0$.  We are required to show that in fact $K_X+B\sim_{\mathbb{Q}}0$.   
Let $Y$ be the normalisation of $X\otimes_K \overline{K}$, with morphism $\phi:Y\to X$.  By \cite[Lemma 2.11]{cascini_relative_2017}, it is sufficient to show that $\phi^*(K_X+B)\sim_{\mathbb{Q}}0$.  For this we show that $\Alb(Y)$ is trivial.  Suppose for contradiction that $\Alb(Y)$ is positive dimensional, and let $f:Y\to T$ be the Stein factorisation of $Y\to\Alb(Y)$. Since $\Alb(Y)$ is non-trivial, $T$ must have positive dimension. If we can find a rational curve which is horizontal over $T$ then we obtain the desired contradiction since there would be a rational curve on $\Alb(Y)$.

First note that there is some finite extension $F/K$ such that if $Y_F$ is the normalisation of $X\otimes_k F$ then $Y=Y_F\otimes_F \overline{K}$, and similarly $f:Y\to T$ descends to $f_F:Y_F\to T_F$.  Let $E$ be the unique subfield of $F$ such that $F/E$ is purely inseparable and $E/K$ is separable.  Now let $V=X_E=X\otimes_K E$, and $\psi:V\to X$ be the natural morphism.  $V$ is a variety since we assume that $K=H^0(X,\sO_X)$.
We also have \[\psi^*(K_X+B)=K_V+B_V\equiv 0\] where $(V,B_V)$ is a klt pair.   
 However, $V$ may no longer be $\mathbb{Q}$-factorial.  If not, simply replace it with a small $\mathbb{Q}$-factorial modification.  Also, since $F/E$ is purely inseparable, \autoref{lem:descend_rational_map} gives a morphism $g:V\to S$ and purely inseparable morphism $T\to S$ fitting into a commutative diagram
 \[
\begin{tikzcd}
Y_F\ar[r]\ar[d,"f_F"] & V\ar[d,"g"]\\
T_F\ar[r] & S
\end{tikzcd}
\]

We claim that $V$ is of log Fano type.  Since $B_V$ is big, there is an effective divisor $E$ and ample $H\geq 0$ with $B_V\sim_{\mathbb{Q}} H+ E$. Now let $\epsilon$ be sufficiently small that $(V,(1-\epsilon)B_V+\epsilon E)$ is klt, and we have $K_V+(1-\epsilon) B_V+\epsilon E\sim_{\mathbb{Q}} -\epsilon H$, which shows that $V$ is of log Fano type.  Hence by the cone theorem \autoref{thm:cone} all of the extremal rays in $\overline{NE}(X)$ are discrete.  Furthermore, \autoref{thm:cone} guaranteed that every extremal ray contains the image of some rational curve on $Y$. 
Let $A$ be the pullback from an ample divisor on $S$.  Then there are some $A$-positive curves in $\overline{NE}(V)$ because $\dim(S)>0$, hence there must be some $A$-positive extremal ray $R$.  Since $V$ is log Fano type, $R$ contains the image of a rational curve from $Y$, which intersects $A$ positively and hence is not contracted over $S$.  Therefore the rational curve on $Y$ is not contracted over $T$, providing the required contradiction.

\end{proof}

In the next result, we show that $K_X+B$ is good, which is to say that the divisor $D$ on the image of its nef reduction map is big.  This will allow us to apply Keel's theorem to it in the next section.  We roughly follow the outline of \cite[Proposition 6.6]{birkar_existence_2017}.

\begin{theorem}\label{thm:it_is_good}
Let $(X,B)$ be a $\mathbb{Q}$-factorial klt pair of dimension $3$ over an uncountable field $K=H^0(X,\sO_X)$.  Assume either that $K$ is $F$-finite or that \autoref{thm:lmm_dw} holds in the non-$F$-finite case.
Suppose that $B$ is a big $\mathbb{Q}$-divisor and $K_X+B$ is nef.  Then there exists a separable extension $L/K$, a projective birational contraction $\phi:W\to X_L$ and projective contraction $h:W\to Z$ with $W$ and $Z$ normal, such that
\begin{enumerate}
\item $\dim(Z)=n(K_X+B)$
\item There is a big and nef divisor $D$ on $Z$ such that  \[\phi^*(K_X+B)\sim_{\mathbb{Q}}h^*D.\]
\end{enumerate}
Hence \[n(K_X+B)=\kappa(K_X+B)=\dim(Z).\]
\end{theorem}

\begin{proof}
First note that whenever we take a nef reduction map, we must additionally perform a separable field extension.  All of the hypotheses are preserved by this process except for the $\mathbb{Q}$-factoriality.  This can be recovered by performing a small $\mathbb{Q}$-factorial modification, which is again harmless, so we will do this without comment where necessary.

If $K_X+B$ is big, the theorem follows from \autoref{prop:max_nef_dim}.  Similarly if $n(K_X+B)=0$ then we already saw that $K_X+B\sim_{\mathbb{Q}}0$ in \autoref{prop:pullback_from_nef_red}.  Now if $n(K_X+B)=1$, then the nef reduction map is a proper morphism since it is a rational map from a normal variety to regular curve which is regular over the generic point.  Therefore $D$ is a numerically non-trivial nef divisor on a curve so is ample and we are done.  Therefore we may assume from now on that $n(K_X+B)=2$, and for contradiction we assume that $\kappa(K_X+B)\leq 1$.

After replacing $(X,B)$ by its base change by a separable field extension and performing a small $\mathbb{Q}$-factorial modification, we may assume that the nef reduction map is $f:X\dashrightarrow Z$ where $\dim(Z)=2$.  Recall that this is regular and proper over the generic point of $Z$, and by \autoref{prop:pullback_from_nef_red} we have morphisms 
\[
\begin{tikzcd}[column sep=tiny, row sep=small]
& W\ar[dl,"\phi",swap]\ar[dr,"h"]&\\
X\ar[rr,"f",dashrightarrow]& & Z
\end{tikzcd}
\]
where \[\phi^*(K_X+B)\sim_{\mathbb{Q}}h^*D\]
and we assume for contradiction that $D$ is not big.

We claim that by replacing $X$, we may assume that every possible sequence of MMP steps which are $K_X+B$-trivial consists only of flips.  Suppose otherwise, so there exists a sequence of flips 
$$\begin{tikzcd}[column sep=tiny, row sep=small]
	X_i\ar[rr,dashed,]\ar[dr,"g_i"] && X_{i+1}\ar[dl, "g_i^+"]\\
	&Z_i&
\end{tikzcd}
$$ 
for some klt pairs $K_{X_i}+\Delta_i$ on $X_i$, each of which is $(K_X+B)$-trivial and lead to a $(K_{X_n}+B_{X_n})$-trivial divisorial contraction $g_n: X_n\to X'$.  By the base-point free theorem, there is some divisor $D_i$ such that $K_{X_i}+B_{X_i}\sim_{\mathbb{Q}}g_i^*D_i$, and so each of the steps is crepant for $K_{X_i}+B_{X_i}$.  Now 
 \autoref{prop:same_nef_red_map} ensures that neither the flips nor the divisorial contraction can be horizontal over $Z$, and so the composition $X'\dashrightarrow X\dashrightarrow Z$ is still the nef reduction map for $K_{X'}+B_{X'}$.
Since $X\dashrightarrow X'$ was crepant for $K_X+B$, we may replace $X$ with $X'$ and restart.  
 Since each such replacement reduces the Picard rank, this process eventually stops, so we may assume that every sequence of $(K_X+B)$-trivial MMP steps consists only of flips.

Let $H_Z$ be an ample divisor on $Z$, and let $H=\phi_*h^*H_Z$.  Since $B$ is big, $B-\gamma H$ is also big for $\gamma$ sufficiently small, and so we can write $B-\gamma H\sim_{\mathbb{Q}} A+E$ for some ample $A\geq 0$ and $E\geq 0$.  Since $(X,B)$ is klt, so is $(X,B+\gamma^2H+\gamma(A+E))$ for $\gamma$ sufficiently small.  
Note that we have \[K_X+B\sim_{\mathbb{Q}}K_X+(1-\gamma) B+\gamma^2H+\gamma (A+E)\] and 
hence we can replace $(X,B)$ by $(X,(1-\gamma)B+\gamma^2H+\gamma(A+E))$, and replace $A$ by $\gamma A$ and $H$ by $\gamma^2 H$ in order to assume that  $B\geq H+A$.

Note that $\phi^*H-h^*H_Z$ is $\phi$-exceptional and $\phi$-antinef and hence is effective by the negativity lemma.  Therefore for any $\epsilon>0$, $K_X+B-\epsilon H$ is not pseudo-effective, since 
\[\phi^*(K_X+B-\epsilon H) \leq \phi^*(K_X+B)-\epsilon h^*H_Z= h^*(D-\epsilon H_Z)\] and $D-\epsilon H_Z$ is not pseudo-effective.

For any fixed and sufficiently small $\epsilon>0$ to be determined, let \[\delta=\min\{t\mid K_X+B-\epsilon H+tA \mathrm{\ is\ }\psef\}\] 
Since $K_X+B-\epsilon A$ is not pseudo-effective, it follows that $\delta>0$ regardless of $\epsilon$.  But we also see that $\delta$ can be made arbitrarily small by shrinking $\epsilon$.  In particular we can ensure that $\epsilon$ is sufficiently small that $(X,B-\epsilon H+\delta A)$ is klt. Note that this is indeed a boundary by our replacement of $B$ earlier.  Furthermore, $K_X+B-\epsilon H+\delta A$ is not big since $\delta$ is the pseudo-effective threshold.

From now on ensure that $\epsilon$ is sufficiently small that \autoref{prop:polytopes}\autoref{itm:polytope_flips} will apply to any sequence of $(K_{X}+B_{X})$-trivial $(K_{X}+B_{}-\epsilon H+\delta A)$-flips.  By \autoref{thm:lmm_dw}, there is a log minimal model $Y$ for $(X,B_{}-\epsilon H_{}+\delta A_{})$.  
 Furthermore, since $(X_{},B_{}-\epsilon H_{}+\delta A_{X})$ is klt, $Y\dashrightarrow X$ contracts no divisors and so $X\dashrightarrow Y$ is an isomorphism away from some closed subset of $Y$ of codimension at least $2$.
We want to take the nef reduction map $g:Y\dashrightarrow V$, but first, we we claim that there is not a sequence $\epsilon_i\to 0$ such that  each corresponding $Y_i$ satisfies $K_{Y_i}+B_{Y_i}-\epsilon_i H_{Y_i}+\delta A_{Y_i}\equiv 0$.

Suppose that we can fix a small $\epsilon$ such that 
\[K_Y+B_Y-\epsilon H_Y+\delta A_Y\equiv 0.\]  
If this holds, we first prove that $X\dashrightarrow Y$ is an isomorphism in codimension $1$.  Suppose otherwise, that it contracts divisors.  Then let $\alpha:U\to Y$ and $\beta: U\to X$ be a common resolution. 
$\alpha^*(K_Y+B_Y-\epsilon H_Y+\delta A_Y)$ is a numerically trivial divisor on $U$, and since $K_X+B$ is klt, if we run a $K_U+B_U+E$-MMP over $X$, where $B_U$ is the strict transform of $B$ and $E$ is the reduced exceptional divisor, this MMP terminates on $X$ by the negativity lemma.  Since $\alpha^*(K_Y+B_Y-\epsilon H_Y+\delta A_Y)$ is numerically trivial, the big and nef base point free theorem \cite[Theorem 1.4]{das_waldron_imperfect} implies that it descends via pull backs through every step of the MMP $U\to X$.  
In other words, $\beta_*\alpha^*(K_Y+B_Y-\epsilon H_Y+\delta A_Y)$ is a numerically trivial $\mathbb{Q}$-Cartier divisor on $X$. 
Now we have that $$(K_X+B-\epsilon H+\delta A)-\beta_*\alpha^*(K_Y+B_Y-\epsilon H_Y+\delta A_Y)=E$$ is effective by pulling this back to $U$ and applying the definition of log minimal model, and is supported on the exceptional locus of $X\dashrightarrow Y$. 
Since $B$ is big we may assume that $B\geq \lambda \Supp(E)$ for some sufficiently small $\lambda\geq 0$, just as we did with $H$ earlier.  
Now run a $(K_X+B-\epsilon H +\delta A)$-MMP.  This MMP eventually ceases contracting curves contained in $\Supp(E)$ by  \cite[Theorem 2.15]{hacon_four}.  But since each contraction is negative for $E$, this means that the MMP must terminate in a log minimal model, which is isomorphic to $Y$ in codimension one.

  By choice of $\epsilon$, the first divisorial contraction must have been $(K_X+B)$-trivial by \autoref{prop:polytopes}\autoref{itm:polytope_flips}, which is impossible since we have already replaced $X$ to rule out this possibility.  Therefore there are no divisorial contractions in this MMP, and so $X$ and $Y$ are isomorphic in codimension $1$.  But this then means that $E=0$ and so $K_X+B-\epsilon H+\delta A\equiv 0$ since it is crepant with $K_Y+B_Y-\epsilon H_Y+\delta A_Y$.  
  We assumed that we could choose a sequence $\epsilon_i\to 0$ so that for each $i$ we obtained  $K_{Y_i}+B_{Y_i}-\epsilon_i H_{Y_i}+\delta_i A_{Y_i}\equiv 0$.  But now we see that we also have $K_X+B-\epsilon_i H+\delta_i A\equiv 0$ for each $i$.
But now taking this limit as $\epsilon_i\to 0$ in this equation, we obtain $K_X+B\equiv 0$, which is impossible since it has nef dimension $2$.  Hence we can assume that $n(K_Y+B_Y-\epsilon H_Y+\delta A_Y)>0$ for all $\epsilon$ sufficiently small.

After a suitable separable base change and small $\mathbb{Q}$-factorialisations of $X$ and $Y$, consider the nef reduction map $g:Y\dashrightarrow V$ of $(Y,B_Y-\epsilon H_Y+\delta A_Y)$.  Since $K_X+B-\epsilon H+\delta A$ was not big we must have $\dim(V)\leq 2$, and we have shown that we may assume that $\dim(V)>0$.

Let $\eta$ be the generic point of $V$ and let 
\[\tau=\max\{t\mid K_{Y_\eta}+B_{Y_\eta}-t H_{Y_\eta} \mathrm{\ is\ }\psef\}.\]
We have $\tau\geq 0$ since $K_X+B$ is pseudo-effective, but also $\tau<\epsilon$ since $$K_{Y_{\eta}}+B_{Y_\eta}-\epsilon H_{Y_\eta}\equiv -\delta A_{Y_{\eta}}$$ and  
$A_Y$ is big since $A$ was ample on $X$.

We aim to show that $\tau=0$, so suppose for contradiction that $\tau>0$.  
We first claim that $(Y,B_Y-\tau H_Y)$ is klt.  By construction of $Y$ we have that $(Y, B_Y-\epsilon H_Y+\delta A_Y)$ is klt, and hence so is $(Y, B_Y-\epsilon H_Y)$.  Recall that $B_Y\geq H_Y+A_Y$, so we may write $B_Y=\Delta+H_Y$.  The coefficients of $\Delta$ and $H_Y$ depend only on those of the original $B$, $A$ and $H$ and not on choices of $\epsilon$, $\delta$ and $Y$.  Therefore by ACC for log canonical thresholds \autoref{thm:acc_for_lct} we can choose $\epsilon$ sufficiently small such that whenever $(Y, \Delta+(1-\epsilon)H_Y)$ is log canonical, then so is $(Y,\Delta+H_Y)$.  But this shows that $(Y, B_Y)$ is log canonical, and hence $(Y, B_Y-\tau H_Y)$ is klt since it lies between the klt $(Y, B_Y-\epsilon H_Y)$ and the log canonical $(Y, B_Y)$.

There is some non-empty open subset $U$ of $V$ over which the nef reduction map is proper.  We may therefore run a $(K_Y+B_Y-\tau H_Y)|_{g^{-1}(U)}$-MMP with scaling of some ample divisor, which terminates terminates with a log minimal model $(Y', B_{Y'}-\tau H_{Y'})$ by \cite[Theorem 9.35]{bhatt_globally}, with morphism $g':Y'\to U$.  But $K_{Y'}+B_{Y'}-\tau H_{Y'}$ is semi-ample over $U$ after possibly shrinking $U$ (by abundance for curves or surfaces \cite{tanaka_abundance_2015} applied to the generic fiber), and not relatively big, so comes with a fibration $h:Y'\to W$ over $U$ by \cite[Theorem 9.33]{bhatt_globally}.  This shows that $\tau=0$ by applying \autoref{thm:global_acc} on the generic fibre with the DCC set being formed by
 taking increasing values of $1-\tau$ created by taking $\epsilon\to 0$.

Given that $\tau=0$ we see that $K_Y+B_Y$ is pseudo-effective over $V$, but not big over $V$.  Furthermore by the application of \autoref{thm:acc_for_lct} above, we see that $(Y, B_Y)$ is log canonical.  We claim that we can still run a terminating $(K_Y+B_Y)|_{g^{-1}(U)}$-MMP over $U$ as above.  Since $g$ is the nef reduction map for $K_Y+B_Y-\epsilon H_Y+\delta A_Y$, which is a klt pair with big boundary, we have \[(K_Y+B_Y-\epsilon H_Y+\delta A_Y)|_{g^{-1}(U)}\sim_{\mathbb{Q},U}0\] by the relative base point free theorem \cite[Theorem 9.33]{bhatt_globally}. 
So for $\lambda\in(0,1)$ we have 
\[(K_Y+B_Y+\lambda(-\epsilon H_Y+\delta A_Y))|_{g^{-1}(U)}\sim_{\mathbb{Q},U}(1-\lambda)(K_Y+B_Y).\]
But since $(Y,B_Y)$ is log canonical and $(Y, B_Y-\epsilon H_Y+\delta A_Y)$ is klt we see that $(Y, B_Y+\lambda(-\epsilon H_Y+\delta A_Y))$ is actually klt. Therefore we can run a $((K_Y+B_Y+\lambda(-\epsilon H_Y+\delta A_Y)))|_{g^{-1}(U)}$-MMP with scaling of some ample divisor which terminates, using  \cite{bhatt_globally}, and note that this was also an MMP for $(K_Y+B_Y)|_{g^{-1}(U)}$ since the two divisors are $\mathbb{Q}$-linearly equivalent over $U$.
Let $(Y',B_{Y'})$ be the resulting log minimal model over $U$.  Since we still have \[K_{Y'}+B_{Y'}+\lambda(-\epsilon H_{Y'}+\delta A_{Y'})\sim_{\mathbb{Q},U}(1-\lambda)(K_{Y'}+B_{Y'})\]
and the former pair is klt, we see that $K_{Y'}+B_{Y'}$ is semi-ample over $U$ by \cite[Theorem 9.33]{bhatt_globally}.  Let $h:Y'\to W$ be the resulting contraction.  Note that since $K_Y+B_Y$ was not big over $V$, we have $\dim(W)<\dim(Y)$.
Since  $(K_Y+B_Y-\epsilon H_Y+\delta A_Y)|_{g^{-1}(U)}$ is relatively trivial over $U$ by \cite[Theorem 9.33]{bhatt_globally}, this property descends through the LMMP and so \[K_{Y'}+B_{Y'}-\epsilon H_{Y'}+\delta A_{Y'}\sim_{\mathbb{Q},W}0.\]  This then implies that \[\delta A_{Y'}-\epsilon H_{Y'}\sim_{\mathbb{Q},W} 0\] since $Y'\to W$ is the semi-ample fibration for $K_{Y'}+B_{Y'}$.

Note that $X\dashrightarrow Y$ is an isomorphism away from a closed subset of $Y$ of codimension at most $2$, and 
$Y|_{g^{-1}(U)}\dashrightarrow Y'$ is an isomorphism away from the union of a closed subset of $Y'$ of codimension at least $2$.  This means that since $\dim(X)=3$, $Y|_{g^{-1}(U)}\dashrightarrow Y'$ is an isomorphism over a general curve in the fibers of $h:Y'\to U$.  Any such curve is proper already in $Y|_{g^{-1}(U)}$ since $g$ is proper over $U$, and so $X\dashrightarrow Y$ is an isomorphism over a general such curve as well.  Therefore
the composition $\psi:X\dashrightarrow Y \dashrightarrow Y'$ is an isomorphism over a general curve in the fibers of $h:Y'\to W$.  Choose such a curve $C$.  
 Let $\widetilde{C}=\psi^{-1}_*C$.  Then 
\[(K_X+B)\cdot\widetilde{C}=(K_{Y'}+B_{Y'})\cdot C=0\]
 since these divisors are the same near $C$.  By the generality of $C$, we can ensure that  $\widetilde{C}$ contains a very general point of $X$, and so is contracted by the nef reduction map $h:X\dashrightarrow Z$.    But since $C$ is contracted over $U$ we also have \[(\epsilon H-\delta A)\cdot \widetilde{C}=(\epsilon H_{Y'}-\delta A_{Y'})\cdot C=0.\] 
 But this is impossible, because $A$ is ample on $X$ and $H$ is the pullback of an ample divisor on $Z$, and so we would have $A\cdot\widetilde{C}=(\epsilon/\delta) H\cdot\widetilde{C}=0$, contradicting the ampleness of $A$.  
 This provides the contradiction to the initial assumption, so we conclude that the divisor $D$ was in fact big.

\end{proof}

\section{Base point freeness and Mori fibre spaces}\label{sec:bpf}

We now prove the main result, the base point free theorem for non-big divisors, and then deduce its consequences for the threefold LMMP.

\begin{lemma}\label{lem:bpf_Q}
Let $(X,B+A)$ be a $\mathbb{Q}$-factorial klt threefold pair over a field $K$ such that $A\geq0$ is an ample $\mathbb{Q}$ divisor and $B$ a $\mathbb{Q}$-boundary. Assume either that $K$ is $F$-finite or that \autoref{thm:lmm_dw} holds in the non-$F$-finite case.

If  $K_X+B+A$ is nef then it  is semi-ample.
\end{lemma}

\begin{proof}

First, we may replace $K$ to assume that $K=H^0(X,\sO_X)$.  Then we may perform an uncountable transcendental extension to assume that $K$ is uncountable, and finally replace $K$ by the finite separable extension $L$ appearing in \autoref{thm:it_is_good}.  Showing that the pullback of $K_X+B+A$ by this field extension is semi-ample is sufficient.  
All assumptions are preserved by these operations except perhaps for $\mathbb{Q}$-factoriality, which can be recovered by performing a small $\mathbb{Q}$-factorial modification followed by perturbing the pullback of $A$. 
In this way we may  assume that
 there are morphisms 
\[
\begin{tikzcd}
X'\ar[r,"\phi"]\ar[d,"h"] &X\\
Z & 
\end{tikzcd}
\]
where $\phi$ is a birational contraction and such that $Z$ carries a big divisor $D$ on $Z$ with 
\[\phi^*(K_X+B+A)\sim_{\mathbb{Q}}h^*D.\]

If $\dim(Z)=0$ we are done, and similarly if $\dim(Z)=1$ then $D$ is ample and we are done.  Also,
if $\dim(Z)=3$ then the result is \autoref{thm:big_bpf}, so we may assume that $\dim(Z)=2$.  

We will first show that 
$K_X+B+A$ is EWM, for which it is sufficient show that $D$ is EWM. By \cite{keel_basepoint_1999} it is enough to check that $D|_{\mathbb{E}(D)}$ is EWM, but $\mathbb{E}(D)$ consists of a union of curves on which $D$ is numerically trivial by the projection formula, so it is EWM to a point as required.  Let $f:X\to V$ be the map associated to $K_X+B+A$.
To check that $D$ is semi-ample, it is enough to check that $D|_{\mathbb{E}(D)}$ is semi-ample by \cite{keel_basepoint_1999}.  For this it is in turn enough to show that $(K_X+B+A)|_F$ is semi-ample for every fibre $F$ of $f$ by the  argument for the moreover of \cite[Lemma 9.26]{bhatt_globally}.
The proof of this given in \cite[Theorem 9.27]{bhatt_globally} also works here.  In particular the relevant argument starts in the third paragraph of the proof of \cite[Theorem 9.27]{bhatt_globally}. 

\end{proof}

\begin{proof}[Proof of \autoref{thm:bpf}]
First, by the same argument as \cite[Theorem 1.4]{das_waldron_imperfect} we can assume that the ground field is $F$-finite.  
Furthermore by taking a small $\mathbb{Q}$-factorialisation \cite[Corollary 9.21]{bhatt_globally} we may assume that $X$ is $\mathbb{Q}$-factorial.
Let $L-(K_X+B)\sim_{\mathbb{R}}A+\epsilon E$ for $A$ ample and $\epsilon$ sufficiently small.  Then by including a small perturbation of $A$ into $E$ we may assume that $A$ is an ample $\mathbb{Q}$-divisor, and since $\epsilon$ is sufficiently small, replacing $B$ by $B+\epsilon E$ results in a klt pair.  Now by \autoref{lem:add_ample} we may replace $A$ up to linear equivalence to assume that $A\geq 0$ and $(X,B+A)$ is a klt pair.  If $B$ is a $\mathbb{Q}$-divisor, we are done by \autoref{lem:bpf_Q}.  Otherwise, by \autoref{prop:polytopes}\autoref{itm:polytope_polytope} applied to the set of all extremal rays,  we may find $\mathbb{Q}$-boundaries $B_i$ and real numbers $b_i>0$ with \(\sum B_i=1\) such that $B=\sum_i b_iB_i$ for which $K_X+B_i+A$ is nef for all $i$.  Therefore $K_X+B_i+A$ is semi-ample for all $i$, and hence $K_X+B+A$ is semi-ample as required.
\end{proof}

\begin{theorem}[Contraction Theorem]\label{thm:contraction}
Let $(X,B)$ be a projective $\mathbb{Q}$-factorial dlt pair of dimension $3$, projective over a field $K$ of characteristic $p>5$. Then any $(K_X+B)$-negative extremal ray can be contracted by a projective morphism.
\end{theorem}
\begin{proof}
Since the ray $R$ is $(K_X+B)$-negative, by a small perturbation of $B$ we can assume that it is a $\mathbb{Q}$-divisor and $(X,B)$ is klt. Since $R$ is also $(K_X+B+A)$-negative for some sufficiently small ample divisor $A$, there are only finitely many $K_X+B+A$-negative rays by \autoref{thm:cone}, each of which contains a curve, we can find an ample $\mathbb{Q}$-divisor $H$ such that $(K_X+B+H)^\perp=R$. Now by replacing $H$ up to linear equivalence such that $(X,B+H)$ is klt we are done by \autoref{lem:bpf_Q}.
\end{proof}

\begin{theorem}[Finiteness of models]\label{thm:finiteness_of_models}
Let $X$ be a three dimensional klt $\mathbb{Q}$-factorial variety over a field $K$.  
Assume either that $K$ is $F$-finite or that \autoref{thm:lmm_dw} holds in the non-$F$-finite case.
  Let $A\geq 0$ be a big $\mathbb{Q}$-divisor and $V$ a finite dimensional rational affine subspace of the vector space of $\mathbb{R}$-Weil divisors on $X$.  Recall the Shokurov polytope \autoref{def:shokurov_polytope}:
\[\mathcal{L}_A(V)=\{\Delta\mid 0\leq (\Delta-A)\in V\mathrm{\ and\ }(X,\Delta)\mathrm{\ is\ log\ canonical}\}\]

Let $\mathcal{C}$ be a rational polytope inside $\mathcal{L}_A(V)$ such that $(X, B)$ is klt for every $B\in \mathcal{C}$.  Then there exist finitely many birational maps $\phi_i:X\dashrightarrow Y_i$ such that for each $B\in \mathcal{C}$ for which  $K_X+B$ is pseudo-effective, there is some $i$ such that $(Y_i, B_{Y_i})$ is a log minimal model of $(X,B)$.
\end{theorem}
\begin{proof}
The proof is identical to that of \cite[Proposition 4.2]{birkar_existence_2017} using \autoref{prop:polytopes}, \autoref{thm:bpf} and \autoref{thm:lmm_dw}.
\end{proof}

\begin{theorem}[Termination with scaling]\label{thm:termination}
Let $(X,B)$ be a three dimensional $\mathbb{Q}$-factorial klt pair, projective over a field $K$ of characteristic $p>5$.  
Assume either that $K$ is $F$-finite or that \autoref{thm:lmm_dw} holds in the non-$F$-finite case.
Suppose also that $B$ is a big, and $C\geq 0$ is such that $K_X+B+C$ is nef.   Then we can run the LMMP for $K_X+B$ with scaling of $C$ and it terminates.  
\end{theorem}

\begin{proof}[Proof of \autoref{thm:termination}]
The proof is the same as \cite[Proposition 4.3]{birkar_existence_2017}, using \autoref{thm:contraction}, \autoref{thm:cone} and \autoref{thm:finiteness_of_models}. 
\end{proof}

\begin{proof}[Proof of \autoref{thm:mfs} in the $F$-finite case]
Fix an ample divisor $A$ such that $K_X+B+A$ is nef.  We claim that we can run a $(K_X+B)$-MMP with scaling of $A$ which terminates.  	We split into cases based on whether $K_X+B$ is pseudo-effective or not.  If it is pseudo-effective, then the argument of \cite[Proposition 4.4]{birkar_existence_2017} applies by replacing \cite[Proposition 3.8]{birkar_existence_2017} with \autoref{prop:polytopes}.  Otherwise, if $K_X+B$ is not pseudo-effective there is $\epsilon>0$ such that $K_X+B+\epsilon A$ is not pseudo-effective.  If we run a $(K_X+B+\epsilon A)$-MMP with scaling of $A$ and if it terminates then the result is also a $(K_X+B)$-Mori fibre space.  This terminates by \autoref{thm:termination} since the boundary $B+\epsilon A$ is big.
\end{proof}

\section{Non-$F$-finite fields}\label{sec:non_F-finite}

In this section, we extend all of our results to non-$F$-finite ground fields.  
As noted in \cite[Remark 1.9]{das_waldron_imperfect}, the assumption appeared there due to a limitation on the existence of log resolutions.  In particular it was not known that there was a projective log resolution with ample exceptional divisor, however this is now known (\cite[Theorem 1]{kollar_witaszek} and \cite[Proposition 2.14]{bhatt_globally}) with the small caveat that the resulting resolution is not an isomorphism over the simple normal crossings locus.   This prevents adding an ample perturbation to the boundary of a dlt pair.

Firstly note that the $\dim(T)>0$ assumption does not appear in \cite[Section 9]{bhatt_globally} until \cite[Proposition 9.20]{bhatt_globally}.  Therefore the only main result of \cite{das_waldron_imperfect} which is not already recovered in the non-$F$-finite case is the existence of log minimal models, and the base point free theorem.  The latter can be deduced immediately from the $F$-finite case since semi-ampleness is preserved under base change, see \cite[Proof of Theorem 1.4]{das_waldron_imperfect}.  While this can probably be obtained as in \cite{das_waldron_imperfect} by appealing to the argument of \cite{birkar_existence_2016}, we give a quick reduction to the $F$-finite case instead.

\begin{proposition}\label{prop:non-f-finite-lmmp}
	Let $(X,B)$ be a $\mathbb{Q}$-factorial klt pair, projective over a field $K$ of characteristic $p>5$, such that $K_X+B$ is pseudo-effective.   Then there is a $(K_X+B)$-MMP which terminates.  In particular $(X,B)$ has a log minimal model.
	\end{proposition}
\begin{proof}
	As in \cite[Section 6.6]{das_waldron_imperfect}, we can find an $F$-finite subfield $L\subset K$ such that $(X,B)$ extends from $L$, which means that there is a $\mathbb{Q}$-factorial klt pair $(X_L,B_L)$ over $L$ such that $$(X,B)\cong (X_L,B_L)\otimes_L K.$$  
	
	By \autoref{thm:termination} we can run a terminating $(X_L, B_L)$-MMP with scaling of some ample divisor $A_L$:
	
\[
\begin{tikzcd}[column sep=tiny, row sep=small]
X_L:=X_1^L\ar[rr,dashrightarrow]\ar[dr] & & X_2^L\ar[dl]\ar[rr,dashrightarrow] & & \dots\ar[rr,dashrightarrow] & & X_{n-1}^L\ar[dr]\ar[rr,dashrightarrow]& & X_n^L\ar[dl]\\
& Z_1^L&& &  & &&Z_{n-1}^L & \\
\end{tikzcd}
\]

Here $X_i^L\to Z_i^L$  denotes either a flipping contraction or a divisiorial contraction, while $X_{i+1}^L\to Z_i^L$ denotes either the flip or the identity respectively.  
	For each $i$, let $X_i=X_i^L\otimes_L K$ and $Z_i=Z_i^L\otimes_L K$ and $A=A_L\otimes_L K$, where by definition we have $X=X_1$. 	
	
	Now run an $(X,B)$-MMP over $Z_1$ with scaling of $A$.  This terminates by \cite[Proposition 9.20]{bhatt_globally} with some model $\widetilde{X}_2$ since $\dim(Z_i)=3$.  
	There is a morphism $\phi:\widetilde{X}_2\to X_2$ since $K_{\widetilde{X}_2}+\widetilde{B}_2$ is nef over $Z_1$ and hence semi-ample over $Z_1$ by the base point free theorem \cite[Theorem 1.4]{das_waldron_imperfect}.  Hence $\phi$ exists by the uniqueness of canonical models.  Note that we still have $K_{\widetilde{X}_2}+\widetilde{B}_2=\phi^*(K_{X_2}+{B}_2)$. 
	Let $\lambda$ be the number which appeared in the LMMP with scaling, which is to say that $\lambda:=\min\{t\mid K_{X}+B+tA\mathrm{\ is\ nef}\}$.  Then by construction, $X\to Z\leftarrow X_2$ is crepant for $K_X+B+\lambda A$, and hence we have $K_{\widetilde{X}_2}+\widetilde{B}_2+\lambda \widetilde{A}=\phi^*(K_{X_2}+{B}_2+\lambda A)$ is nef as well.  It follows that $\widetilde{A}=\phi^*A$ and so the next nef threshold on $X_2$ is the same as the next nef threshold on $\widetilde{X}_2$.   

	Continuing this process produces:
	
\[
\begin{tikzcd}[column sep=tiny, row sep=small]
X=X_1\ar[rr,dashrightarrow]\ar[d] & & \widetilde{X}_2\ar[d]&\dots & & \widetilde{X}_{n-1}\ar[d]\ar[rr,dashrightarrow]& & \widetilde{X}_n\ar[d]\\
X=X_1\ar[rr,dashrightarrow]\ar[dr]& & X_2\ar[dl]&\dots& & X_{n-1}\ar[dr]\ar[rr,dashrightarrow]& & X_n\ar[dl]\\
& Z_1&&\dots & &&Z_{n-1} & \\
\end{tikzcd}
\]
	
	where the top row is a $(K_{X}+B)$-MMP with scaling of $A$.  Eventually we find that $K_{\widetilde{X}_n}+\widetilde{B}_n=\phi_n^*(K_{{X}_n}+{B}_n)$ is nef and so this LMMP has terminated.  The pair $(\widetilde{X}_n,\widetilde{B}_n)$ was obtained by running a $K_X+B$-MMP which terminated on $(\widetilde{X}_n,\widetilde{B}_n)$ and hence $(\widetilde{X}_n,\widetilde{B}_n)$ is a log minimal model for $(X,B)$. 
	
	\end{proof}

At this point, all of the results which were conditional on \autoref{thm:lmm_dw} holding in the $F$-finite case are unconditional.

\bibliographystyle{acm}
\bibliography{library}

\end{document}